\documentclass[11pt,letterpaper]{amsart}

\usepackage{amsthm}
\usepackage{amssymb}
\usepackage{amsmath}
\usepackage{graphicx}
\usepackage{amscd}
\usepackage{amsfonts}
\usepackage{amsbsy}
\usepackage[T1]{fontenc}
\usepackage[english]{babel}

\makeatletter
\@namedef{subjclassname@2020}{\textup{2020} Mathematics Subject Classification}
\makeatother

\usepackage{epsfig}
\usepackage{amssymb}

\newtheorem{thm}{Theorem}

\newtheorem{definition}{Definition}[section]
\newtheorem{lemma}[thm]{Lemma}

\newtheorem{ex}{Example}[section]

\title[Expansiveness for flows on noncompact spaces]{Expansiveness for flows on noncompact spaces}

\author{Y. Yang}

\address{School of Mathematics and Statistics, Liaoning University, Shenyang, 110036, China.}
\email{yangyinong@lnu.edu.cn}

\author{C.A. Morales}

\address{Hangzhou International Innovation Institute of Beihang University,  
Hangzhou 311115, China.}
\email{morales@impa.br}


\begin{document}

\begin{abstract} The classical notion of expansive flow \cite{bw} is invariant under topological conjugacy on compact spaces but not in general \cite{lny}. Then, we introduce the concept of {\em topological expansive flow}. We prove that this concept is conjugacy-invariant and coincides with expansivity on compact spaces. Moreover, a flow is topological expansive if and only if it is rescaling expansive \cite{jny} and its singularities are isolated points of the space.
Finally, we prove that the growth rate of periodic orbits for topological expansive flows is controlled by an entropy-like invariant. This extends Bowen–Walters inequality \cite{bw}.
\end{abstract}

\maketitle

\section{Introduction}

\noindent
A fundamental problem in dynamics is to determinate which dynamical properties are preserved under topological conjugacy or topological equivalence.
Certain classical properties—such as minimality, transitivity, and the nonwandering property—do. In the compact setting, additional properties like expansivity, shadowing, and topological stability also do. However, these latter properties fail to be invariant outside the compact case \cite{lny}.

In this paper, we focus on the expansivity for flows on metric spaces
introduced by Bowen and Walters \cite{bw}. Though these authors proved that this property is invariant under topological equivalence on compact spaces, it is not so
on general metric spaces. Then, we use the methods in \cite{lny} to introduce the notion of {\em topological expansive flow}. We show that this definition is not only invariant under topological conjugacy but also agrees with expansivity in the compact case. We further prove that a flow is topological expansive if and only if it is rescaling expansive (as defined in \cite{jny}) and its singularities are isolated points of the space.
Finally, we estimate the growth rate of periodic orbits for topological expansive flows through an entropy-like invariant, extending the Bowen–Walters inequality \cite{bw}.
We now state our results precisely.

A {\em flow} on a metric space $X$ is a one-parameter family of continuous maps $\phi = \{\phi_t : X \to X\}_{t \in \mathbb{R}}$ such that $\phi_0 = \mathrm{id}_X$, $\phi_{s+t} = \phi_s \circ \phi_t$, and the map $(x, t) \mapsto \phi_t(x)$ is continuous. For $I \subset \mathbb{R}$, we define the orbit segment $\phi_I(x) = \{\phi_t(x) : t \in I\}$. The full orbit of $x \in X$ is $\phi_{\mathbb{R}}(x)$. An {\em orbit} is any such set.
Two flows $\phi$ on $X$ and $\varphi$ on $Y$ are {\em topologically conjugate} if there exists a homeomorphism $h : Y \to X$ such that $\phi_t \circ h = h \circ \varphi_t$ for all $t$. If $h$ only maps orbits of $\varphi$ onto those of $\phi$, the flows are said to be {\em topologically equivalent}.
A property of flows is said to be {\em invariant under topological conjugacy} (resp. {\em equivalence}) if it is preserved under topologically conjugate (resp. equivalent) flows. Since conjugacy implies equivalence, the latter invariance implies the former.

On the other hand, a flow $\phi$ on $X$ is {\em expansive} if for every $\epsilon > 0$ there exists $\delta > 0$ such that if $x, y \in X$ and a continuous map $\alpha : \mathbb{R} \to \mathbb{R}$ with $\alpha(0) = 0$ satisfy
$d(\phi_t(x), \phi_{\alpha(t)}(y)) < \delta \quad \text{for all } t\in\mathbb{R},$
then $y \in \phi_{[-\epsilon,\epsilon]}(x)$. This property is invariant under topological equivalence on compact metric spaces (Corollary 4 in \cite{bw}) but not under conjugacy in general. A counterexample can be obtained from Example 1 of \cite{jny}. Another counterexample is as follows.

\begin{ex}
\label{jimmy}
Let $\phi$ be the flow on $X=\{(x,y,z)\in \mathbb{R}^3\mid z=0, (x,y)\neq(0,0)\}$ defined by $\phi_t(w) = e^tw$ for all $w\in X$ and $t\in\mathbb{R}$.
It follows that $\phi$ is expansive.
Let $\varphi$ be its conjugate on the punctured sphere $Y = S^2 \setminus \{\mbox{the poles}\}$ via stereographic projection. It follows that $\varphi$ is not expansive.
\end{ex}

These examples motivated \cite{jny} to introduce the following definitions:
Let $\mathrm{Sing}(\phi)$ denote the set of {\em singularities} i.e. points $x\in X$ with $\phi_t(x)=x$ for all $t\in\mathbb{R}$. Define
$$
C_\phi(X) = \{ \delta:X\to [0,\infty): \delta\mbox{ is continuous and }\delta^{-1}(0) = \mathrm{Sing}(\phi) \}.
$$
\begin{definition}
A flow $\phi$ on $X$ is {\em rescaling expansive} if
for every $\epsilon > 0$ there exists $\delta \in C_\phi(X)$ such that if $x, y \in X$ and a continuous map $\alpha : \mathbb{R} \to \mathbb{R}$ with $\alpha(0) = 0$ satisfy
$d(\phi_t(x), \phi_{\alpha(t)}(y)) \leq \delta(\phi_t(x))$ for all $t\in\mathbb{R},$
then $y \in \phi_{[-\epsilon,\epsilon]}(x)$.
\end{definition}

Though rescaling expansivity is invariant under topological conjugacy (Proposition 1 in \cite{jny}) it does not reduce to expansivity on compact metric spaces.
A counterexample is as follows.

\begin{ex}
\label{lor}
Let $X$ be the {\em geometric Lorenz attractor} (or "strange strange attractor" in \cite{g}). It is a compact subset of $\mathbb{R}^3$ endowed with a flow $\phi$. It follows that $\phi$ is rescaling expansive in the sense of \cite{wew} and so
rescaling expansive as defined above. On the other hand, $\phi$ is not expansive since it has a singularity which is not an isolated point of $X$ (see \cite{bw}).
\end{ex}

Now, we introduce a notion of expansivity for flows that is both invariant under topological conjugacy and reduces to expansivity on compact metric spaces.
We will use the techniques in \cite{lny} based on previous works by Dowker \cite{d} and Hurley \cite{h}.
Define
$$
C^+(X)=\{\delta:X\to (0,\infty):\delta\mbox{ is continuous }\}
$$
\begin{definition}
\label{def1}
A flow $\phi$ on $X$ is {\em topological expansive} 
if for every $\epsilon > 0$ there exists $\delta \in C^+(X)$ such that if $x, y \in X$ and a continuous map $\alpha : \mathbb{R} \to \mathbb{R}$ with $\alpha(0) = 0$ satisfy
$d(\phi_t(x), \phi_{\alpha(t)}(y)) < \delta(\phi_t(x)) \quad \text{for all } t\in\mathbb{R},$
then $y \in \phi_{[-\epsilon,\epsilon]}(x)$.
\end{definition}

This is the flow version of what was called "expansive homeomorphism" in \cite{lny}. We adopt the term "topological expansive" (following \cite{c}) rather than "expansive", as used in \cite{lny} because of the comparisons in the examples below.

This generalization of expansivity replaces the constant $\delta$ in the definition of expansivity with a continuous positive function. Every expansive flow is therefore topological expansive. The converse is not true by the
example below. A metric space $X$ is {\em discrete} if for every $x\in X$ there is $\rho_x\in (0,\infty)$ such that $B(x,\rho_x)=\{x\}$.
If we can choose $\rho_x$ to be independent on $x$, then we say that
$X$ is {\em uniformly discrete}.

\begin{ex}
\label{go}
The trivial flow $\phi$ (i.e. $\phi_t=id_X$ for all $t\in\mathbb{R}$) is topological expansive if and only if $X$ is discrete.
On the other hand, $\phi$ is expansive if and only if $X$ is uniformly discrete.
\end{ex}

In particular, the trivial flow of a discrete but not uniformly discrete space is topological expansive but not expansive. Further examples of this type will be given below.

We now state that the notion of topological expansive flow corresponds to our expectations.

\begin{thm}
\label{thA}
Topological expansivity is invariant under topological conjugacy and reduces to expansivity on compact metric spaces. More precisely, a flow of a compact metric space is topological expansive if and only if it is expansive.
\end{thm}

From this theorem we derive the following example.

\begin{ex}
\label{patin}
The flow $\varphi$ from Example \ref{jimmy} is topological expansive by Theorem \ref{thA}
since it is conjugate to the expansive (hence topological expansive) flow $\phi$. However, $\varphi$ is not expansive.
\end{ex}

Another example of a topological expansive flow which is not expansive is motivated by Example 2.10 in \cite{c}:

\begin{ex}
\label{jim}
Let $\phi$ be a {\em translation flow} of a Banach space $X$, i.e.,
there is $v\in X\setminus\{0\}$ such that $\phi_t(x)=x+tv$ for all $x\in X$ and $t\in\mathbb{R}$. Then, $\phi$ is topological expansive but not expansive.
\end{ex}

\begin{proof}
First we prove that $\phi$ is topological expansive.
Given $\epsilon>0$ define
$\delta:X\to(0,\infty)$ by
$$
\delta(x)=\frac{1}2\|v\|\epsilon e^{-\|x\|},\quad\quad\forall x\in X.
$$
Then, $\delta(x)<\|v\|\epsilon$ for all $x\in X$ and
$\delta(x)\to0$ as $\|x\|\to\infty$.

Now, suppose that $x,y\in X$ and a continuous function $s:\mathbb{R}\to\mathbb{R}$ fixing $0$ satisfy
\begin{equation}
\label{barraza}
\|\phi_t(x)-\phi_{s(t)}(y)\|<\delta(\phi_t(x)),\quad\quad\forall t\in\mathbb{R}.
\end{equation}
Clearly $\phi_t$ is an isometry, so
$\|x-\phi_{s(t)-t}(y)\|\leq \delta(\phi_t(x))$ for all $t\in\mathbb{R}$ thus
$$
(s(t)-t)v\to x-y\quad\mbox{ as }\quad
t\to\pm\infty.
$$
But $(s(t)-t)v\in \langle v\rangle$ (the subspace generated by $v$) which is closed so $x-y\in \langle v\rangle$ hence
$y=x+\alpha v$ for some $\alpha\in\mathbb{R}$.
It follows that
$$
y=\phi_\alpha(x).
$$
To estimate $\alpha$, we just replace $t=0$ in \eqref{barraza} to get
$$
\|x-y\|<\delta(x)<\|v\|\epsilon.
$$
Then, $\|x-(x+\alpha v)\|<\|v\|\epsilon$ so
$$
|\alpha|<\epsilon
$$
thus $y\in \phi_{[-\epsilon,\epsilon]}(x)$ whence $\phi$ is topological expansive.

Finally, we see that $\phi$ is not expansive since it is a flow by isometries of a Banach space.
This completes the proof.
\end{proof}

One more application of Theorem \ref{thA} is as follows.
A $C^1$ vector field $V$ of a Riemannian manifold is topological expansive if its induced flow is topological expansive.

\begin{ex}
\label{colina}
The vector field $V(x,y)=(1, y)$ of $\mathbb{R}^2$ is topological expansive.
\end{ex}

\begin{proof}
The induced flow $\phi$ is given by the solutions of the ODE's in $\mathbb{R}^2$,
$$
\left\{
\begin{array}{rcl}
\dot{x}=1\\
\dot{y}=y.
\end{array}
\right.
$$
By solving it we obtain $\phi_t(x,y)=(x+t,e^ty)$ for $x,y,t\in \mathbb{R}$.
Let $\varphi$ be the flow of $\mathbb{R}^2$ defined by
$\varphi_t(x,y)=(x+t,y)$ for all $x,y,t\in\mathbb{R}$. Then, $\varphi$ is a translation flow
(just take $v=(1,0)$) and so topological expansive by Example \ref{jim}.
Now, define $h:\mathbb{R}^2\to\mathbb{R}^2$ by $h(x,y)=(x,e^xy)$ for all $x,y\in \mathbb{R}$. Then, $h$ is a homeomorphism and  $\phi_t\circ h=h\circ \varphi_t$ for all $t\in\mathbb{R}$ so $\phi$ and $\varphi$ are topologically conjugated.
Then, $\phi$ (and so $V$) are topological expansive by Theorem \ref{thA}.
\end{proof}

Now, we compare topological expansivity and rescaling expansivity.

\begin{thm}
\label{don}
A flow of a metric space $X$ is topological expansive if and only if it is rescaling expansive and its singularities are isolated points of $X$.
\end{thm}

In particular, not every rescaling expansive flow is topological expansive (Example \ref{lor}).

Next, we deal with the growth rate of the periodic orbits namely
$$
\limsup_{t\to\infty}\frac{1}t\log v(t),
$$
where $v(t)$ denotes the number of periodic orbits with period less than or equal to $t$ of the given flow $\phi$.

Recall that a point $x$ is {\em periodic} if $\phi_t(x) = x$ for some minimal $t > 0$ (called period) while a {\em periodic orbit} is the full orbit of a periodic point. Clearly, every periodic orbit is formed by periodic points all of them with common period. Such a common value is then referred to as the {\em period of the periodic orbit}.

The motivation is Bowen-Walters inequality \cite{bw}
$$
\limsup_{t\to\infty}\frac{1}t\log v(t)\leq e(\phi)
$$
which holds for expansive flows $\phi$ on compact metric spaces.
Here $e(\phi)$ is the topological entropy of $\phi$ defined by $e(\phi)=h(\phi_1)$ where $h(f)$ is the Adler-Konheim-McAndrew topological entropy for maps $f:X\to X$.
Equivalently, $e(\phi)$ is given by the Bowen-Dinaburg formula,
$$
e(\phi)=\lim_{\epsilon\to0}\limsup_{t\to\infty}\frac{1}t\log S(t,\epsilon),
$$
where $S(t,\epsilon)$ is the maximal cardinality of a $(t,\epsilon)$-separating set i.e. $E\subset X$ such that for all distinct $x,y\in E$ there is $0\leq s\leq t$ such that $d(\phi_t(x),\phi_t(y))\geq \epsilon$.

The question is if an analogous inequality holds for the topological expansive flows.
To handle this question we introduce the following entropy-like number
$$
e^*(\phi) = \sup_K e^*(\phi, K),
$$
where the supremum is over compact sets $K \subset X$, and
$$
e^*(\phi, K) = \sup_{\delta \in C_+(X)} \limsup_{t \to \infty} \frac{1}{t} \log \left(\frac{S(t, \delta, K)}{\beta(t,\delta,K)}\right),
$$
where
$$
\beta(t,\delta,K)=\inf\{\delta(\phi_s(z))\mid (z,s)\in K\times [0,t]\}
$$
and $S(t, \delta, K)$ denoting the maximal cardinality of a {\em $(t, \delta, K)$-separating set}. The latter are subsets $E\subset K$ satisfying that for all distinct $x,y\in E$ there is $0\leq s\leq t$ such that
$d(\phi_s(x),\phi_s(y))\geq \delta(\phi_s(x))$.
Sometimes we use the notations $S(t, \delta, K,\phi)$ and $\beta(t, \delta, K,\phi)$
to indicate dependence of $\phi$.

To justify the entropy-like nature of $e^*(\phi)$ we will prove the following result (resembling some features of the topological entropy on compact metric spaces).

\begin{thm}
\label{sa}
The properties below hold for all flows $\phi$ and $\varphi$ of metric spaces $X$ and $Y$ and all $a\in (0,\infty)$:
\begin{enumerate}
\item
If $\phi$ and $\varphi$ are topologically conjugated, then $e^*(\phi)=e^*(\varphi)$.
\item
If $X$ is compact, then $e^*(\phi)=e(\phi)$.
\item
If $X=Y$ and $\varphi_t=\phi_{at}$, then $e^*(\varphi)=ae^*(\phi)$.
\item
Let $R(t,\delta,K)$ be the minimal cardinality of a {\em $(t,\delta,K)$-spanning set} i.e.
$F\subset X$ such that $\forall y\in K$ $\exists x\in F$ such that $d(\phi_s(x),\phi_s(y))\leq \delta(\phi_s(x))$ for all $0\leq s\leq t$. Then,
$$
e^*(\phi)=\sup_K\sup_{\delta\in C^+(X)}\limsup_{t\to\infty}\frac{1}t\log \left(\frac{R(t,\delta,K)}{\beta(t,\delta,K)}\right).
$$
\end{enumerate}
\end{thm}

The first two items say that $e^*(\phi)$ is invariant under topological conjugacy and reduces to the topological entropy in the compact case.
Item (3) is a version of the power formula for the topological entropy for maps
while the last item says that $e^*(\phi)$ satisfies the so-called Bowen-Dinaburg formula
\cite{b}, \cite{di}.

Finally, we say that $\phi$ is {\em dynamically isolated at infinity} if there is a compact subset $K\subset X$ such that
$$
\bigcap_{t\in\mathbb{R}}\phi_t(X\setminus K)\subset Sing(X).
$$
Equivalently, if there is a compact subset $K\subset X$ intersection all regular (i.e. nonsingular) orbits of $\phi$. Obviously every flow of a compact metric space is dynamically isolated at infinity. It is easy to construct flows which are not dynamically isolated at infinity.

With these notations we can state the following result.

\begin{thm}
\label{thB}
Let $\phi$ be a topological expansive flow of a metric space.
If $\phi$ is dynamically isolated at infinity, then
$$
\limsup_{t \to \infty} \frac{1}{t} \log v(t) \leq e^*(\phi).
$$
\end{thm}

It follows from theorems \ref{thA} and \ref{sa} that the above inequality reduces to Bowen-Walters's inequality \cite{bw} on compact metric spaces.

The remainder of the paper is organized as follows.
In Section \ref{sec2}, we prove some preliminary facts.
In Section \ref{sec3}, we prove the theorems.

\section{Preliminary lemmas}
\label{sec2}

\noindent
As is well-known, every singularity of an expansive flow of a metric space is an isolated point of the space \cite{bw}.
The analogous fact holds for topological expansive flows.

\begin{lemma}
\label{alacran}
Every singularity of a topological expansive flow $\phi$ of a metric space $X$ is an isolated point of $X$.
\end{lemma}

\begin{proof}
Given $\sigma\in Sing(\phi)$ take $\delta\in C^+(X)$ from the definition for $\epsilon=1$. Hence, if $x\in B(\sigma,\delta(\sigma))$,
the continuous map $\alpha:\mathbb{R}\to\mathbb{R}$ fixing $0$ defined by
$\alpha(t)=0$ for all $t\in\mathbb{R}$ satisfies
$d(\phi_t(\sigma),\phi_{\alpha(t)}(x))=d(\sigma,x)<\delta(\sigma)=\delta(\phi_t(\sigma))$ for all $t\in\mathbb{R}$ so $x\in \phi_{[-1,1]}(\sigma)=\{\sigma\}$ thus
$B(\sigma,\delta(\sigma))=\{\sigma\}$ proving the result.
\end{proof}

Given a flow $\phi$ of a metric space $X$, we say that $A\subset X$ is {\em invariant} if $\phi_t(A)=A$ for all $t\in\mathbb{R}$. In such a case, the restriction of $\phi$ to $A$ is the flow $\phi|_A$ of $A$ defined by
$(\phi|_A)_t(a)=\phi_t(a)$ for all $a\in A$.

The next lemma gives a sufficient condition for topological expansivity when restricted to the set of regular orbits.

\begin{lemma}
\label{meta}
If $\phi$ is a rescaling expansive flow of a metric space $X$, then
$\phi|_{X\setminus Sing(\phi)}$ is topological expansive.
\end{lemma}

\begin{proof}
Let $\epsilon>0$ and $\hat{\delta}\in C_\phi(X)$ be given by the rescaling expansivity of $\phi$. Then, $\delta=\hat{\delta}|_{X\setminus Sing(\phi)}\in C^+(X\setminus Sing(\phi))$.
Let $x,y\in X\setminus Sing(\phi)$ and $s:\mathbb{R}\to\mathbb{R}$
be a continuous function fixing $0$ such that
$$
d((\phi|_{X\setminus Sing(\phi)})_t(x),(\phi|_{X\setminus Sing(\phi)})_{s(t)}(y))<\delta((\phi|_{X\setminus Sing(\phi)})_t(x)),
\qquad\forall t\in\mathbb{R}.
$$
Then, $d(\phi_t(x),\phi_{s(t)}(y))\leq\hat{\delta}(\phi_t(x))$ for all $t\in\mathbb{R}$
and so $y\in \phi_{[-\epsilon,\epsilon]}(x)$ by the choice of $\hat{\delta}$.
This completes the proof.
\end{proof}

Another sufficient condition for topological expansivity is given below.

\begin{lemma}
\label{chuzo}
If $\phi$ is a flow of a metric space $X$ exhibiting a disjoint union
$X=A\cup B$ formed by closed invariant subset, then $\phi$ is topological expansive if and only if $\phi|_A$ and $\phi|_B$ are topological expansive.
\end{lemma}

\begin{proof}
It is easy to prove that the topological expansivity is hereditary namely
the restriction of a topological expansive flow to any invariant set of it is topological expansive.
This implies the necessity.
For the sufficiency we argue as follows.

Suppose that both $\phi|_A$ and $\phi|_B$ are topological expansive.
Given $\epsilon>0$ let $\delta_A\in C^+(A)$ and $\delta_B\in C^+(B)$ be given by the topological expansivity of $\phi|_A$ and $\phi|_B$ respectively.
By Lemma 1 in \cite{h} there is $\hat{\delta}\in C^+(X)$ such that
$$
B(a,\hat{\delta}(a))\cap B=\emptyset \quad(\forall a\in A)\quad\mbox{ and }\quad B(b,\hat{\delta}(b))\cap A=\emptyset\quad (\forall b\in B).
$$
Defining
$$
\delta(x) = \left\{
\begin{array}{rcl}
\min\{\delta_A(x),\hat{\delta}(x)\}, & \mbox{if} & x\in A\\
& & \\
\min\{\delta_B(x),\hat{\delta}(x)\},& \mbox{if} & x\in B
\end{array}
\right.
$$
we have a map $\delta:X\to (0,\infty)$.
We prove that $\delta$ is continuous.
Let $x_n\in X$ be a sequence converging to some $x\in X$.
Then, $x$ belongs to either $A$ or $B$. In the first case,
$x_n\in B(x,\hat{\delta}(x))$ for all $n$ large so $x_n\in A$ for $n$ large
hence
$\phi(x_n)=\min\{\delta_A(x_n),\hat{\delta}(x_n)\}\to\min\{\delta_A(x),\hat{\delta}(x)\}=\phi(x)$.
Likewise, $\phi(x_n)\to\phi(x)$ if $x\in B$ thus $\phi$ is continuous.

Now, suppose that $x,y\in X$ and a continuous map $s:\mathbb{R}\to\mathbb{R}$ fixing $0$ satisfy
$d(\phi_t(x),\phi_{s(t)}(y))<\delta(\phi_t(x))$ for all $t\in\mathbb{R}$.
Again either $x\in A$ or $x\in B$.
In the first case $\phi_t(x)\in A$ (since $A$ is invariant) hence
$\phi_{s(t)}(y)\in A$ for all $t\in\mathbb{R}$
since $\delta\leq\hat{\delta}$. It follows that
$d((\phi|_A)_t(x),(\phi|_A)_{s(t)}(y))<\delta_A((\phi|_A)_t(x))$ for all $t\in\mathbb{R}$ so $y\in (\phi|_A)_{[-\epsilon,\epsilon]}(x)$ thus
$y\in \phi_{[-\epsilon,\epsilon]}(x)$. Likewise, $y\in \phi_{[-\epsilon,\epsilon]}(x)$ whenever $x\in B$ proving the result.
\end{proof}

Next, we introduce the Dowker lemma.
Recall that a map $\beta:X\to\mathbb{R}$,
where \(X\) is a topological space, is said to be \emph{lower semicontinuous} (respectively, \emph{upper semicontinuous}) if for every \(x\in X\) and every \(\epsilon\in (0,\infty)\) there exists a neighborhood \(U\) of \(x\) such that 
\[
\beta(y)>\beta(x)-\epsilon \quad (\text{respectively, } \beta(y)<\beta(x)+\epsilon)
\]
for every \(y\in U\). Given a metric space $X$ and \(\rho, \gamma\in C^+(X)\) we write  
\(\gamma < \rho\) whenever \(\gamma(x) < \rho(x)\) for all \(x \in X\).

\begin{lemma}[\emph{Dowker lemma}, \cite{d}, \cite{h} or p. 172 of \cite{k}]
\label{dl}
Let \(X\) be a metric space. If \(\beta,\gamma:X\to \mathbb{R}\) are, respectively, lower semicontinuous and upper semicontinuous functions and \(\gamma<\beta\), then there exists a continuous function \(\alpha:X\to\mathbb{R}\) such that $\gamma<\alpha<\beta$.
\end{lemma}

A first application of the Dowker lemma is as follows.

\begin{lemma}
\label{pura}
Let $\phi$ be a flow all of whose singularities are isolated points of $X$. Then,
there is $r\in C^+(X)$ such that
$$
\bigcap_{t\in\mathbb{R}}\phi_t(B(x,r(x))\subset \{x\},\quad\quad\forall x\in X.
$$
\end{lemma}

\begin{proof}
We can assume that there is $x_*\in X$ such that
$$
\Lambda_*=\bigcap_{t\in\mathbb{R}}\phi_t(B(x_*,1))\neq\emptyset.
$$
Otherwise, the constant function $r(x)=1$ ($\forall x\in X$) satisfies the conclusion of the lemma.

Since every $x\in Sing(\phi)$ is an isolated point,
for all such $x$ there is $\rho_x>0$ such that
$B(x,\rho(x))=\{x\}$.

Define
$$
\beta(x) = \left\{
\begin{array}{rcl}
\inf\left\{r>0\mid \displaystyle\bigcap_{t\in\mathbb{R}}\phi_t(B(x,r))\neq\emptyset\right\}, & \mbox{if} & x\in X\setminus Sing(\phi)\\
& & \\
\rho_x,& \mbox{if} & x\in Sing(X).
\end{array}
\right.
$$

We have $\beta(x)<\infty$ for all $x\in X$. In fact,
given $x\in X$ there is $r_x\in (0,\infty)$ large enough such that $B(x,r_x)$ covers $\Lambda_*$.
Then,
$$
\emptyset\neq\Lambda_*\subset \bigcap_{t\in\mathbb{R}}\phi_t(B(x,r_x))
$$
and so $\beta(x)\leq r_x<\infty$.
Clearly $ \beta(x)\geq0$ for all $x\in X$.
If $\beta(x)=0$ for some $x\in X$, then $x\in X\setminus Sing(\phi)$ and so
there is a sequence $y_n\in X$ such that
$d(\phi_t(y_n),x)<\frac{1}n$ for all $n\in\mathbb{N}$ and $t\in\mathbb{R}$.
In particular, $y_n\to x$ as $n\to\infty$ and also $\phi_t(y_n)\to x$ as $n\to\infty$ for all $t\in \mathbb{R}$.
Then, for all $t\in\mathbb{R}$ one has
$$
d(\phi_t(x),x)\leq d(\phi_t(x),\phi_t(y_n))+d(\phi_t(y_n),x)\to0\quad\mbox{ as }\quad n\to\infty
$$
so $\phi_t(x)=x$ for all $t\in\mathbb{R}$ thus $x\in Sing(\phi)$ which is absurd.
Therefore, $\beta(x)>0$ for all $x\in X$.

In this way, we have a map $\beta:X\to(0,\infty)$.
Let us prove that this map is lower semicontinuous.
Otherwise, there are $\epsilon>0$, $x\in X$ and a sequence $x_n\to x$ such that
$\beta(x)-\epsilon\geq \beta(x_n)$ for all $n\in\mathbb{N}$.
Take $r=\beta(x)-\frac{\epsilon}2$ then $\beta(x_n)<r$ for all $n\in\mathbb{N}$ so
there is a sequence $y_n\in X$ such that $d(\phi_t(y_n),x_n)<r$
for all $n\in\mathbb{N}$ and $t\in\mathbb{R}$
thus
$$
d(\phi_t(y_n),x)\leq d(\phi_t(y_n),x_n)+d(x_n,x)<r+d(x_n,x),\quad\quad\forall n\in\mathbb{N}, \, t\in\mathbb{R}.
$$
Since $x_n\to x$ and $r<\beta(x)-\frac{\epsilon}4$, we can choose $n$ large such that
$d(\phi_t(y_n),x)<\beta(x)-\frac{\epsilon}4$ for all $t\in\mathbb{R}$.
Then,
$$
y_n\in \bigcap_{t\in\mathbb{R}}\phi_t(B(x,\beta(x)-\frac{\epsilon}4))
\quad\mbox{ and so }\quad
\bigcap_{t\in\mathbb{R}}\phi_t(B(x,\beta(x)-\frac{\epsilon}4))\neq\emptyset
$$
contradicting the definition of $\beta(x)$. Therefore, $\beta$ is lower semicontinuous.

Then, by Dowker lemma applied to this $\beta$ and $\gamma=0$ we obtain
$r:X\to \mathbb{R}$ continuous such that $0<r<\beta$.
We have $r\in C^+(X)$ by the first inequality.

If $x\in Sing(\phi)$, $r(x)<\rho_x$ hence
$$
\bigcap_{t\in\mathbb{R}}\phi_t(B(x,r(x))\subset \bigcap_{t\in\mathbb{R}}\phi_t(B(x,\rho_x)=\{x\}.
$$
If $x\in X\setminus Sing(\phi)$, one has
$$
\bigcap_{t\in\mathbb{R}}\phi_t(B(x,r(x))=\emptyset\subset \{x\}
$$
by the definition of $\beta(x)$.
Thus, $r$ satisfies the desired property.
\end{proof}

A second application is as follows.

\begin{lemma}
\label{mcoy}
Let $\phi$ be a topological expansive flow of a metric space.
Then, for every $\delta\in C^+(X)$ there is $\alpha\in C^+(X)$ such that
\begin{equation}
\label{sistema}
\sup_{|\eta|\leq \alpha(x)}d(\phi_\eta(x),x)< \delta(x),\quad\quad\forall x\in X.
\end{equation}
\end{lemma}

\begin{proof}
By Lemma \ref{alacran} we have that every singularity of $\phi$ is an isolated point of $X$.
Then, we can choose $r\in C^+(X)$ satisfying Lemma \ref{pura}.
Given $\delta\in C^+(X)$ we define $\hat{\delta}=\min\{r,\delta\}$ hence
$\hat{\delta}\in C^+(X)$ too.
Since every $x\in Sing(X)$ is an isolated point of $X$, for all such $x$ there is $\rho_x\in (0,\infty)$ such that $B(x,\rho_x)=\{x\}$.

For every $x\in X$ we define
$$
\beta(x) = \left\{
\begin{array}{rcl}
\sup\left\{\gamma>0\mid \displaystyle\sup_{|\eta|\leq \gamma}d(\phi_\eta(x),x)< \hat{\delta}(x)\right\}, & \mbox{if} & x\in X\setminus Sing(\phi)\\
& & \\
\rho_x,& \mbox{if} & x\in Sing(X).
\end{array}
\right.
$$
Clearly $\beta(x)>0$ for all $x\in X$.
Let us prove $\beta(x)<\infty$ for all $x\in X$.
Suppose not namely $\beta(x)=\infty$ for some $x\in X$.
It follows that
$\phi_\eta(x)\in B(x,r(x))$ for all $\eta\in\mathbb{R}$ so Lemma \ref{pura} implies
$$
x\in \bigcap_{\eta\in\mathbb{R}}\phi_\eta(B(x,r(x))\subset\{x\}.
$$
It follows that $x\in Sing(\phi)$ so $\beta(x)=\rho_x\in (0,\infty)$ which is absurd.
Therefore, $\beta(x)<\infty$.

In this way, we obtain a map $\beta:X\to (0,\infty)$.
We assert that this map is lower semicontinuous.

Otherwise, there are $x\in X$, a sequence $x_n\to x$ and $\epsilon>0$ such that
$$
\beta(x_n)\leq \beta(x)-\epsilon,\quad\quad\forall n\in\mathbb{N}.
$$
If $x\in Sing(\phi)$, we would have $x_n\in B(x,\rho_x)$ for all $n$ large so
$x_n=x$ for all $n$ large yielding $\rho_x\leq \rho_x-\epsilon$ which is absurd.
Then, $x\in X\setminus Sing(\phi)$ and we can further assume $x_n\in X\setminus Sing(\phi)$ for all $n\in\mathbb{N}$.
Since $\beta(x)-\epsilon<\beta(x)-\frac{\epsilon}2$, we can choose a number $\gamma$ such that
$$
\beta(x_n)<\gamma<\beta(x)-\frac{\epsilon}2,\quad\quad
\forall n\in\mathbb{N}.
$$
Then, the definition of $\beta(x_n)$ provides a sequence $\eta_n$ with $|\eta_n|\leq \gamma$ for all $n\in\mathbb{N}$ such that
$$
d(\phi_{\eta_n}(x_n),x_n)\geq \hat{\delta}(x_n),\quad\quad\forall n\in\mathbb{N}.
$$
Clearly, we can assume that $\eta_n\to \eta$ as $n\to\infty$ for some $|\eta|\leq \gamma$.
Since $\hat{\delta}$ is continuous, by letting $n\to\infty$ above we get
$$
d(\phi_\eta(x),x)\geq \hat{\delta}(x).
$$
However, $0<\gamma<\beta(x)$ so
$d(\phi_\eta(x),x)<\hat{\delta}(x)$ by the definition of $\beta(x)$ a contradiction. Therefore, $\beta$ is lower semicontinuous.

Then, we can apply Dowker lemma to this $\beta$ and $\gamma=0$ to obtain
$\alpha\in C^+(X)$ with $0<\alpha<\beta$.
Clearly \eqref{sistema} holds when $x\in Sing(X)$.
If $x\in X\setminus Sing(\phi)$, then $\phi_\eta(x)\in B(x,\hat{\delta}(x))\subset B(x,\delta(x))$ for all $|\eta|\leq \alpha(x)$ (since $\alpha(x)<\beta(x)$) thus \eqref{sistema} holds in this case as well. This completes the proof.
\end{proof}

Next, we introduce the following notation.
Given maps \(\rho, \gamma\in C^+(X)\) we write $\gamma\ll\rho$ whenever
$x\in X$ and $y\in B(x,\gamma(x))$ implies $\gamma(x)<\rho(y)$.
Clearly $\gamma\ll\rho$ implies $\gamma<\rho$.
With this notation we can reformulate Lemma 2.8 in \cite{lny} as follows.

\begin{lemma}
\label{power}
For every metric space $X$ and $\rho\in C^+(X)$ there exists $\gamma\in C^+(X)$ such that $\gamma\ll\rho$.
\end{lemma}

We use this lemma in the proof of the following one (compare with Item (vi) of Theorem 3 in \cite{bw}):

\begin{lemma}
\label{p1}
For every topological expansive flow $\phi$ of a metric space $X$
and every $\epsilon>0$ there is $\alpha\in C^+(X)$ with the following property:
if $(t_i)_{i\in\mathbb{Z}}$ and $(u_i)_{i\in\mathbb{Z}}$ are bi-infinite sequences of real numbers with $t_0=u_0=0$ and $t_i\to\pm\infty$ as $i\to\pm\infty$, and if
if $x,y\in X$ satisfy
$0<t_{i+1}-t_i\leq\alpha(\phi_{t_i}(x))$, $|u_{i+1}-u_i|\leq\alpha(\phi_{u_i}(y))$ and 
$d(\phi_{t_i}(x),\phi_{u_i}(y))\leq\alpha(\phi_{t_i}(x))$ for all $i\in\mathbb{Z}$,
then $y\in \phi_{[-\epsilon,\epsilon]}(x)$.
\end{lemma}

\begin{proof}
Let $\epsilon>0$ and $\delta\in C^+(X)$ be given by the topological expansivity of $\phi$.
By  Lemma \ref{power} there is $\hat{\delta}\in C^+(X)$ such that
$
\hat{\delta}\ll \frac{1}3\delta.
$
Again Lemma \ref{power} together with Lemma \ref{mcoy} provides $\hat{\alpha}\in C^+(X)$ such that
$
\hat{\alpha}\ll \frac{1}2\hat{\delta}
$
and
\begin{equation}
\label{cha}
\sup_{|\eta|\leq \hat{\alpha}(z)}d(\phi_\eta(z),z)<\hat{\delta}(z),\quad\quad\forall z\in X.
\end{equation}

Choose $\check{\delta}\in C^+(X)$ with $\check{\delta}\ll \hat{\delta}$ and, as before, $\check{\alpha}\in C^+(X)$ such that
$
\check{\alpha}\ll \frac{1}2\check{\delta}
$
and
\begin{equation}
\label{chacho}
\sup_{|\eta|\leq \check{\alpha}(z)}d(\phi_\eta(z),z)<\check{\delta}(z),\quad\quad\forall z\in X.
\end{equation}
Finally, by Lemma \ref{power} once more, we can choose $\alpha\in C^+(X)$ such that $\alpha\ll\min\{\hat{\alpha},\check{\alpha}\}$.

Now, let $(t_i)_{i\in\mathbb{Z}}$ and $(u_i)_{i\in\mathbb{Z}}$ be bi-infinite sequences of real numbers with $t_0=u_0=0$ and $t_i\to\pm\infty$ as $i\to\pm\infty$, and let $x,y\in X$ be satisfying
$0<t_{i+1}-t_i\leq\alpha(\phi_{t_i}(x))$, $|u_{i+1}-u_i|\leq\alpha(\phi_{u_i}(y))$ and 
$d(\phi_{t_i}(x),\phi_{u_i}(y))\leq\alpha(\phi_{t_i}(x))$ for all $i\in\mathbb{Z}$.

Since $0<t_{i+1}-t_i$ for all $i\in\mathbb{Z}$ and $t_i\to\pm\infty$ as $i\to\pm\infty$, we can define a (piecewise linear) continuous map
$s:\mathbb{R}\to\mathbb{R}$ by setting $s(t_i)=u_i$ and extending linearly on $[t_i,t_{i+1})$,
for all $i\in\mathbb{Z}$.
Since $t_0=u_0=0$, this map fixes $0$.
Moreover, since $|u_{i+1}-u_i|\leq\alpha(\phi_{u_i}(y))$ we also have $|s(t)-u_i|\leq \alpha(\phi_{u_i}(y))$ for all $i\in\mathbb{Z}$ and $t\in [t_i,t_{i+1})$.

Fix $t\in\mathbb{R}$. Then, $t\in [t_i,t_{i+1})$ for some (actually unique) $i\in\mathbb{Z}$. Clearly $|t-t_i|\leq \alpha(\phi_{t_i}(x))$.
It follows that
\begin{eqnarray}
\label{perra}
d(\phi_t(x),\phi_{s(t)}(y))&\leq& d(\phi_t(x),\phi_{t_i}(x))+d(\phi_{s(t)}(y),\phi_{u_i}(y))+ d(\phi_{t_i}(x),\phi_{u_i}(y))\\
\notag &\leq& d(\phi_{t-t_i}(\phi_{t_i}(x)),\phi_{t_i}(x))+d(\phi_{s(t)-u_i}(\phi_{u_i}(y)),\phi_{u_i}(y))+\\
\notag & & \alpha(\phi_{t_i}(x))\\
\notag &\leq& \sup_{|\eta|\leq \alpha(\phi_{t_i}(x))}d(\phi_\eta(\phi_{t_i}(x)),\phi_{t_i}(x))+\\
\notag & & \sup_{|\eta|\leq \alpha(\phi_{u_i}(y))}d(\phi_\eta(\phi_{u_i}(y)),\phi_{u_i}(y))+\alpha(\phi_{t_i}(x))\\
\notag &\leq& \sup_{|\eta|\leq \hat{\alpha}(\phi_{t_i}(x))}d(\phi_\eta(\phi_{t_i}(x)),\phi_{t_i}(x))+\\
\notag & & \sup_{|\eta|\leq \hat{\alpha}(\phi_{u_i}(y))}d(\phi_\eta(\phi_{u_i}(y)),\phi_{u_i}(y))+\alpha(\phi_{t_i}(x))\\
\notag &\overset{\eqref{cha}}{\leq}& \hat{\delta}(\phi_{t_i}(x))+ \hat{\delta}(\phi_{u_i}(y))+\alpha(\phi_{t_i}(x)).
\end{eqnarray}

In the sequel, we estimate the three summands on the right-hand side above separately.

For the first one, we have $|t-t_i|\leq \alpha(\phi_{t_i}(x))<\hat{\alpha}(\phi_{t_i}(x))$ and then $\phi_t(x)\in B(\phi_{t_i}(x),\hat{\delta}(\phi_{t_i}(x)))$ by \eqref{cha}.
But $\hat{\delta}\ll \frac{1}3\delta$ so
\begin{equation}
\label{perro}
\hat{\delta}(\phi_{t_i}(x))<\frac{1}3\delta(\phi_t(x)).
\end{equation}

We estimate the second summand as follows.
Since $d(\phi_{t_i}(x),\phi_{u_i}(y))<\alpha(\phi_{t_i}(x))<\hat{\alpha}(\phi_{t_i}(x))$,
one has $\hat{\alpha}(\phi_{t_i}(x))<\frac{1}2\hat{\delta}(\phi_{u_i}(y))$
by $\hat{\alpha}\ll\frac{1}2\hat{\delta}$. Then,
$$
d(\phi_{t_i}(x),\phi_{u_i}(y))<\frac{1}2\hat{\delta}(\phi_{u_i}(y)).
$$
Also,
$|t-t_i|<\alpha(\phi_{t_i}(x))<\check{\alpha}(\phi_{t_i}(x))$ so
$
d(\phi_t(x),\phi_{t_i}(x))<\check{\delta}(\phi_{t_i}(x))
$
by \eqref{chacho}. But $d(\phi_{t_i}(x),\phi_{u_i}(y))<\alpha(\phi_{t_i}(x))<\check{\alpha}(\phi_{t_i}(x))<\check{\delta}(\phi_{t_i}(x))$ so
$\check{\delta}(\phi_{t_i}(x))<\frac{1}2\hat{\delta}(\phi_{u_i}(y))$ by $\check{\delta}\ll\frac{1}2\hat{\delta}$. Then,
$$
d(\phi_t(x),\phi_{t_i}(x))<\frac{1}2\hat{\delta}(\phi_{u_i}(y)).
$$
It follows that
\begin{eqnarray*}
d(\phi_t(x),\phi_{u_i}(y))&\leq& d(\phi_t(x),\phi_{t_i}(x))+d(\phi_{t_i}(x),\phi_{u_i}(y))\\
&<& \frac{1}2\hat{\delta}(\phi_{u_i}(y))+\frac{1}2\hat{\delta}(\phi_{u_i}(y))\\
&=& \hat{\delta}(\phi_{u_i}(y)).
\end{eqnarray*}
This together with $\hat{\delta}\ll\frac{1}3\delta$ implies
\begin{equation}
\label{perruno}
\hat{\delta}(\phi_{u_i}(y))<\frac{1}3\delta(\phi_t(x)).
\end{equation}

To estimate the last summand, since $\alpha\ll\hat{\alpha}\ll\hat{\delta}$, we also have
$\alpha<\hat{\delta}$ so
$\alpha(\phi_{t_i}(x))<\hat{\delta}(\phi_{t_i}(x))$ thus
\begin{equation}
\label{perrita}
\alpha(\phi_{t_i}(x))<\frac{1}3\delta(\phi_t(x)).
\end{equation}

Combining \eqref{perra}, \eqref{perro}, \eqref{perruno} and \eqref{perrita} we obtain
$$
d(\phi_t(x),\phi_{s(t)}(y))<\delta(\phi_t(x)),\quad\quad\forall 
t\in \mathbb{R}.
$$
Therefore, $y\in \phi_{[-\epsilon,\epsilon]}(x)$ by the choice of $\delta$.
This completes the proof.
\end{proof}

We finish this section by proving two lemmas to be used in the proof of Item (4) of Theorem \ref{sa}. They are flow-versions of the ones for maps in \cite{ylm}.

\begin{lemma}
\label{le2}
For every flow $\phi$ of a metric space $X$ and every $\delta\in C^+(X)$ there is $\delta'\in C^+(X)$ such that
\begin{equation}
\label{rado}
\frac{S(t,\delta,K)}{\beta(t,\delta,K)}\leq \frac{R(t,\delta',K)}{\beta(t,\delta',K)}
\end{equation}
for all compact $K\subset X$ and $t>0$.
\end{lemma}

\begin{proof}
For the given $\delta$ by Lemma \ref{power} we choose $\gamma\in C^+(X)$ with $\gamma \ll\delta$. We have $\frac{1}4\gamma\in C^+(X)$ so we can choose
$\delta'\in C^+(X)$ such that $\delta' \ll\frac{1}4\gamma$.
Take a compact $K$ and $t>0$.

First, we observe that $\delta'<\delta$ so
\begin{equation}
\label{jura}
\beta(t,\delta',K)\leq \beta(t,\delta,K).
\end{equation}
Next, we let $E$ and $F$ be $(t,\delta,K)$-separating and $(t,\delta',K)$-spanning sets respectively.
We assert that for every $y\in F$ there exists at most one $x\in E$ such that
$d(\phi_s(x),\phi_s(y))\leq \delta'(\phi_s(y))$ for every $0\leq s\leq t$.

Otherwise, there are distinct $x,x'\in E$ such that
$d(\phi_s(x),\phi_s(y))\leq \delta'(\phi_s(y))$
and $d(\phi_s(x),\phi_s(y))\leq \delta'(\phi_s(y))$ for every $0\leq s\leq t$.

Then, by Remark 3.2 in \cite{lny},
$d(\phi_s(x),\phi_s(y))\leq \frac{1}4\gamma(\phi_s(x))$
and
$$
d(\phi_s(x'),\phi_s(y))\leq \frac{1}4\gamma(\phi_s(x')),\qquad\forall 0\leq s\leq t
$$
so
\begin{eqnarray*}
d(\phi_s(x),\phi_s(x'))&\leq& d(\phi_s(x),\phi_s(y))+d(\phi_s(x'),\phi_s(y))\\
&\leq& \frac{\gamma(\phi_s(x))}4+\frac{\gamma(\phi_s(x'))}4\\
&<&\max\{\gamma(\phi_s(x)),\gamma(\phi_s(x'))\}\\
&<& \delta(\phi_s(x)),\quad\quad\quad \forall 0\leq s\leq t,
\end{eqnarray*}
by Remark 3.2 in \cite{lny}. This contradicts that $E$ is $(t,\delta,K)$-separating proving the assertion.

On the other hand, since $F$ is $(t,\delta',K)$-spanning,
for all $x\in E$ we can choose $y\in F$ such that
$d(\phi_s(x),\phi_s(y))\leq \delta'(\phi_s(y))$ for every $0\leq s\leq t$.
This results in a map $x\in E\mapsto y\in F$ which is injective by the assertion. It follows that $E$ has at most the same quantity of elements of $F$ and this implies
$S(t,\delta,K)\leq R(t,\delta',K)$. This together with \eqref{jura} implies \eqref{rado}.
\end{proof}

The second lemma is as follows.

\begin{lemma}
\label{le1}
For every flow $\phi$ of a metric space $X$ and every
$\delta\in C^+(X)$ there is $\delta'\in C^+(X)$ such that
\begin{equation}
\label{silve}
\frac{R(t,\delta,K)}{\beta(t,\delta,K)}\leq \frac{S(t,\delta',K)}{\beta(t,\delta',K)}
\end{equation}
for all compact $K\subset X$ and $t>0$.
\end{lemma}

\begin{proof}
Again for the given $\delta$ by Lemma \ref{power} we choose $\delta\in C^+(X)$ with
choose $\delta' \ll\delta$.
In particular, $\delta'<\delta$.
Take a compact $K$ and $t>0$.
Then, \eqref{jura} holds.

Now, fix a $(t,\delta',K)$-separating set $E$ of cardinality $S(t,\delta',K)$.
If $y\in K\setminus E$, then $E\cup \{y\}$ is not $(t,\delta',K)$-separating due to the maximality of the cardinality of $E$.
Then, there are distinct $z,w\in E\cup\{y\}$ such that
$$
d(\phi_s(z),\phi_s(w))\leq \epsilon'(\phi_s(z)),\quad\quad\forall
0\leq s\leq t.
$$
Since $E$ is $(t,\delta',K)$-separating, we have $y\in \{z,w\}$.
First assume $y=z$.
Then, $w\in E$ and $d(\phi_s(y),\phi_s(w))\leq \delta'(\phi_s(y))$ for all $0\leq s\leq t$.
So, $d(\phi_s(w),\phi_s(y))\leq \epsilon(\phi_s(w))$ for all $0\leq s\leq t$ by Remark 3.2 in \cite{lny} since $\delta' \ll\delta$.
If $y=w$, then $z\in E$ and $d(\phi_s(z),\phi_s(y))\leq \delta'(\phi_s(z))<\delta(\phi_s(z))$ for all
$0\leq s\leq t$ since $\delta'<\delta$. All together prove that $E$ is $(t,\delta,K)$-spanning yielding $R(t,\delta,K)\leq S(t,\delta',K)$. This together with \eqref{jura} implies \eqref{silve}.
\end{proof}

\section{Proof of the theorems}
\label{sec3}

\begin{proof}[Proof of Theorem \ref{thA}]
Let $\phi$ be a topological expansive flow of a metric space $X$.

First assume that $\phi$ is topologically conjugate to another flow $\varphi$ of a metric space $Y$. Then,
there is a homeomorphism $h:Y\to X$ such that
$\phi_t\circ h=h\circ \varphi_t$ for all $t\in\mathbb{R}$.
Fix $\epsilon>0$ and let $\delta\in C^+(X)$ be given by the topological expansivity of $\phi$.
By Lemma 2.7 in \cite{lny} there is $\hat{\delta}\in C^+(Y)$ such that
$$
y,y'\in Y\quad\mbox{ and }\quad d(y,y')<\hat{\delta}(y)\quad\implies\quad d(h(y),h(y'))<\delta(h(y)).
$$
Let $y,y'\in Y$ and a continuous map $s:\mathbb{R}\to\mathbb{R}$ fixing $0$ to satisfy
$$
d(\varphi_t(y),\varphi_{s(t)}(y'))<\hat{\delta}(\varphi_t(y)),\qquad \forall t\in\mathbb{R}.
$$
Then, $d(h\circ \varphi_t(y),h\circ\varphi_{s(t)}(y'))<\delta(h\circ\varphi_t(y))$
and so $d(\phi_t(h(y)),\phi_{s(t)}(h(y')))<\delta(\phi_t(h(y)))$ for all $t\in\mathbb{R}$.
It follows from the choice of $\delta$ that
$h(y')\in \phi_{[-\epsilon,\epsilon]}(h(y))$ so
$h(y')\in h(\varphi_{[-\epsilon,\epsilon]}(y))$ thus
$y'\in\varphi_{[-\epsilon,\epsilon]}(y)$. Therefore, $\varphi$ is topological expansive too.

Now, assume that $X$ is compact and again fix $\epsilon>0$.
Let $\delta\in C^+(X)$ be given by the topological expansivity of $\phi$.
Since $X$ is compact, there is a positive real number $\alpha$ such that
$\alpha<\delta(z)$ for all $z\in X$.
If $x,y\in X$ and a continuous map $s:\mathbb{R}\to\mathbb{R}$ fixing $0$ satisfy
$d(\phi_t(x),\phi_{s(t)}(y))<\alpha$ for all $t\in\mathbb{R}$,
then $d(\phi_t(x),\phi_{s(t)}(y))<\delta(\phi_t(x))$ for all $t\in\mathbb{R}$
thus $y\in \phi_{[-\epsilon,\epsilon]}(x)$ proving that $\phi$ is expansive.
This completes the proof.
\end{proof}

\begin{proof}[Proof of Theorem \ref{don}]
Let $\phi$ be a flow of a metric space $X$.
We shall prove that $\phi$ is topological expansive if and only if $\phi$ is rescaling expansive and the singularities of $\phi$ are isolated points of $X$.

To prove the necessity, assume that $\phi$ is topological expansive.
Then, every singularity is an isolated point by Lemma \ref{alacran}.
To prove that $\phi$ is rescaling expansive,
let $\epsilon>0$ and $\delta\in C^+(X)$ be given by the topological expansivity of $\phi$.
Define $\hat{\delta}:X\to [0,\infty)$ by
$$
\hat{\delta}(x) = \left\{
\begin{array}{rcl}
\frac{\delta(x)}2, & \mbox{if} & x\in X\setminus Sing(\phi)\\
& & \\
0,& \mbox{if} & x\in Sing(X).
\end{array}
\right.
$$
Since $\delta$ is continuous and every singularity is isolated, $\hat{\delta}$ is continuous. Since $\delta>0$, we conclude that $\hat{\delta}\in C_\phi(X)$.
Let $x,y\in X$ and $s:\mathbb{R}\to\mathbb{R}$ be a continuous function fixing $0$ such that
$$
d(\phi_t(x),\phi_t(y))\leq \hat{\delta}(\phi_t(x)),\quad\quad\forall t\in \mathbb{R}.
$$
If $x\in Sing(\phi)$, $x=y$ and so $y\in \phi_{[-\epsilon,\epsilon]}(x)$.
Otherwise, $d(\phi_t(x),\phi_t(y))<\delta(\phi_t(x))$ for all $t\in \mathbb{R}$ thus
$y\in\phi_{[-\epsilon,\epsilon]}(x)$ by the choice of $\delta$.
Therefore, $\phi$ is rescaling expansive.

To prove the sufficiency, assume that $\phi$ is rescaling expansive and that every singularity is an isolated point of $X$.
Then, $\phi|_{X\setminus Sing(\phi)}$ is topological expansive by Lemma \ref{meta}. We have a disjoint union $X=(X\setminus Sing(\phi))\cup Sing(\phi)$ and, since every singularity is isolated, this union is formed by closed subset. Both $X\setminus Sing(\phi)$ and $Sing(\phi)$ are invariant,
and $\phi|_{Sing(\phi)}$ is topological expansive by Example \ref{go} since
$\phi|_{Sing(\phi)}$ is the trivial flow.
Therefore, $\phi$ is topological expansive by Lemma \ref{chuzo}.
\end{proof}

\begin{proof}[Proof of Theorem \ref{sa}]
First we prove Item (1).
By hypothesis there is a homeomorphism $h:Y\to X$ such that $\phi_t\circ h=h\circ \varphi_t$ for all $t\in\mathbb{R}$.
Fix $\delta\in C^+(Y)$. By Lemma 2.7 in \cite{lny} there is $\hat{\delta}\in C^+(X)$ such that
$$
x,x'\in X\quad\mbox{ and }\quad d(x,x')\leq \hat{\delta}(x)\quad\implies\quad d(h^{-1}(x),h^{-1}(x'))\leq \delta(h^{-1}(x)).
$$
By replacing $\hat{\delta}$ by $\frac{1}2\min\{\hat{\delta}, \delta\circ h^{-1}\}$ if necessary, we can assume $\hat{\delta}< \delta\circ h^{-1}$.

Fix a compact $K\subset Y$. Then, $h(E)\subset X$ is compact.
Given $t>0$ let $E$ be a $(t,\delta,K)$-separating set of $\varphi$.
Pick distinct $x,x'\in h(E)$. Then, there are distinct $y,y'\in E$ and $0\leq s\leq t$ such that $h(y)=x$, $h(y')=x'$ and $d(\varphi_s(y),\varphi_s(y'))>\delta(\varphi_s(y))$.
The choice of $\hat{\delta}$ implies
$d(h(\varphi_s(y)),h(\varphi_s(y'))>\hat{\delta}(h(\varphi_s(y)))$
so $d(\phi_s(x),\phi_s(x'))>\hat{\delta}(\phi_s(x))$.
Since $h(E)\subset h(K)$, we conclude that $h(E)$ is $(t,\hat{\delta},h(K))$-separating for $\phi$. Therefore,
$$
S(t,\delta,K,\varphi)\leq S(t,\hat{\delta},h(K),\phi),\quad\quad\forall t>0.
$$
On the other hand, for all $0\leq s\leq t$ one has
$\hat{\delta}(\phi_s(x))=\hat{\delta}(\phi_s(h(y)))=\hat{\delta}(h(\varphi_s(y)))\leq \delta(\varphi_s(y))$ so
$$
\beta(t,\hat{\delta},h(K),\phi)\leq \beta(t,\delta,K,\varphi),\quad\quad\forall t>0.
$$
It follows that
$$
\frac{S(t,\delta,K,\varphi)}{\beta(t,\delta,K,\varphi)}\leq \frac{S(t,\hat{\delta},h(K),\phi)}{\beta(t,\hat{\delta},h(K),\phi)},\quad\quad\forall t>0.
$$
This implies $e^*(\varphi,K)\leq e^*(\phi)$ for all compact $K\subset Y$
so $e^*(\varphi)\leq e^*(\phi)$. Reversing the roles of $\varphi$ and $\phi$ we obtain
$e^*(\phi)\leq e^*(\varphi)$ hence $e^*(\phi)=e^*(\varphi)$.
Item (1) is proved.

To prove Item (2), we first note that $e^*(\phi)=e^*(\phi,X)$ since $X$ is compact.
Next, that every positive constant $\epsilon$ can be seen as a constant function $\epsilon\in C^+(X)$. In such a case $S(t,\epsilon)=S(t,\epsilon,X)$ and
$\beta(t,\epsilon,X)=\epsilon$ for all $t>0$ so
\begin{eqnarray*}
e(\phi)&=&\lim_{\epsilon\to0}\frac{1}t\log S(t,\epsilon)\\
&=&\sup_{\epsilon>0}\limsup_{t\to\infty}\frac{1}t\log \left(\frac{S(t,\epsilon,X)}{\beta(t,\epsilon,X)}\right)\\
&\leq& \sup_{\delta\in C^+(X)}\limsup_{t\to\infty}\frac{1}t\log\left(\frac{S(t,\delta,X)}{\beta(t,\delta,X)}\right)\\
&=&e^*(\phi,X)\\
&=&e^*(\phi).
\end{eqnarray*}
For the reversed inequality,
let $\delta\in C^+(X)$ and $E\subset X$ be $(t,\delta,X)$-separating for some $t>0$.
Since $X$ is compact, there is a positive number $\rho$ such that
$\delta(z)>\rho$ for all $z\in X$.
Thus, $E$ is $(t,\rho)$-separating yielding
$S(t,\delta,X)\leq S(t,\rho)$ and $\beta(t,\delta,X)\geq \rho$ thus
$$
\limsup_{t\to\infty}\frac{1}t\log\left(\frac{S(t,\delta,X)}{\beta(t,\delta,X)}\right)
\leq \limsup_{t\to\infty}\frac{1}t\log S(t,\rho)\leq e(\phi).
$$
By taking the supremum over $\delta\in C^+(X)$ above we get $e^*(\phi,X)\leq e(\phi)$ hence
$e^*(\phi)\leq e(\phi)$. Item (2) is proved.

To prove Item (3) we fix a compact $K\subset X$ and $\delta\in C^+(X)$.
If $E$ is a $(t,\delta,K)$-separating set for $\varphi$ and $x,y\in E$ are distinct, there exists $0\leq s\leq t$ such that $d(\varphi_s(x),\varphi_s(y))>\delta(\varphi_s(x))$.
Then, $d(\phi_{s'}(x),\phi_{s'}(y))>\delta(\phi_{s'}(x))$ with $0\leq s'=as\leq at$
so $E$ is $(at,\delta,K)$-separating thus
$$
S(t,\delta,K,\varphi)\leq S(at,\delta,K,\phi).
$$
Likewise,
\begin{eqnarray*}
\beta(t,\delta,K,\varphi)&=&\inf\{\delta(\varphi_s(z))\mid (z,s)\in K\times [0,t]\}\\
&=&\inf\{\delta(\phi_{s'}(z))\mid (z,s')=(z,as)\in K\times [0,at]\}\\
&=&\beta(at,\delta,K,\phi)
\end{eqnarray*}
so
\begin{eqnarray*}
e^*(\varphi,K)&=&\sup_{\delta\in C^+(X)}\limsup_{t\to\infty}\frac{1}t\log\left(\frac{S(t,\delta,K,\varphi)}{\beta(t,\delta,K,\varphi)}
\right)\\
&\leq& a\sup_{\delta\in C^+(X)}\limsup_{t\to\infty}\frac{1}t\log \left(\frac{S(t,\delta,K,\phi)}{\beta(t,\delta,K,\phi)}\right)\\
&=& ae^*(\phi,K).
\end{eqnarray*}
Since $K$ is an arbitrary compact, $e^*(\varphi)\leq ae^*(\phi)$.
Since $\phi_t=\varphi_{a^{-1}t}$ for all $t\in\mathbb{R}$, we obtain
$e^*(\phi)\leq a^{-1}e^*(\varphi)$ hence $e^*(\varphi)=ae^*(\phi)$.
Item (3) is proved.

To prove Item (4) fix $\delta\in C^+(X)$. For this $\delta$ we let
$\delta'\in C^+(X)$ be given by Lemma \ref{le1}. Then, 
\begin{eqnarray*}
\limsup_{t\to\infty}\frac{1}t\log\left(\frac{S(t,\delta,K)}{\beta(t,\delta,K)}\right)&\leq&
\limsup_{t\to\infty}\frac{1}t\log\left(\frac{R(t,\delta',K)}{\beta(t,\delta',K)}\right)\\
&\leq&
\sup_{K'}\sup_{\delta''\in C^+(X)}\limsup_{t\to\infty}\frac{1}t\log \left(\frac{R(t,\delta'',K')}{\beta(t,\delta'',K')}\right)
\end{eqnarray*}
for all compact $K\subset X$.
By taking the supremum over $\delta$ and $K$ above we obtain
$$
e^*(\phi)\leq\sup_K\sup_{\delta\in C^+(X)}\limsup_{t\to\infty}\frac{1}t\log \left(\frac{R(t,\delta,K)}{\beta(t,\delta,K)}\right).
$$

To prove the reversed inequality,
we fix again $\delta\in C^+(X)$. For this $\delta$ we take $\delta'\in C^+(X)$ given by Lemma \ref{le2}.
Then,
\begin{eqnarray*}
\limsup_{t\to\infty}\frac{1}t\log\left(\frac{R(t,\delta,K)}{\beta(t,\delta,K)}\right)&\leq&
\limsup_{t\to\infty}\frac{1}t\log\left(\frac{S(t,\delta',K)}{\beta(t,\delta',K)}\right)\\
&\leq& \sup_{K'}\sup_{\delta''\in C^+(X)}\limsup_{t\to\infty}\frac{1}t\log \left(\frac{S(t,\delta'',K')}{\beta(t,\delta'',K')}\right)\\
&=& e^*(\phi)
\end{eqnarray*}
for all compact $K\subset X$.
By taking the supremum over $\delta$ and $K$ above we obtain the desired inequality.
Item (4) is proved.
\end{proof}

\begin{proof}[Proof of Theorem \ref{thB}]
Let $\phi$ be a topological expansive flow of a metric space $X$.
Suppose that $\phi$ is dynamically isolated at infinity.
Then, there is a compact $K\subset X$ intersecting all regular orbits of $\phi$.

We now follow the arguments in \cite{bw} with the aid of Lemma \ref{p1}.
Take $\epsilon=1$ and let $\alpha\in C^+(X)$ be given by Lemma \ref{p1}.
Given $t,\rho>0$ we let $v_\rho(t)$ be the number of periodic orbits with period in $[t-\rho,t+\rho]$.
Denote by $[R]$ the integer part of $R\in \mathbb{R}$.

We claim that
\begin{equation}
\label{poro}
v(t_*, {\frac{\beta(t,\alpha,K)}2})\leq S(t,\alpha,K),\quad\quad\forall t>0,\, 0\leq t_*\leq t.
\end{equation}

Indeed, fix $t>0$ and $0\leq t_*\leq t$. For simplicity, we write $\beta=\beta(t,\alpha,K)$.
By selecting one point for each periodic orbits with period in $[t_*-\frac{\beta}2,t_*+\frac{\beta}2]$ we form a subset $E$ with $card(E)=v(t_*,{\frac{\beta}2})$. Since $K$ intersects every regular orbit (and periodic orbits are regular), we can assume $E\subset K$.

Let us prove that $E$ is $(t,\alpha,K)$-separating.
Otherwise, since $E\subset K$, there would exist distinct $x,y\in E$ such that
$$
d(\phi_s(x),\phi_s(y))<\alpha(\phi_s(x)),\quad\quad\forall 0\leq s\leq t.
$$
Let $a$ and $b$ be the periods of $x$ and $y$ respectively so
$a,b\in [t_*-\frac{\beta}2,t_*+\frac{\beta}2]$, $\phi_a(x)=x$ and $\phi_b(y)=y$.

Define $m=[\frac{t_*-\frac{\beta}2}\beta]$ and the sequences
$(t_i)_{i\in\mathbb{Z}},(u_i)_{i\in\mathbb{Z}}$ by
$$
t_i=pa+q\beta\quad\mbox{ and }\quad u_i=pb+q\beta
$$
whenever $i=pm+q$ for some $p\in \mathbb{Z}$ and $0\leq q<m$.

Since $0\leq q\beta\leq t$ for any integer $q=0,\cdots, m-1$,
one has
$$
d(\varphi_{t_i}(x),\varphi_{u_i}(y))=d(\varphi_{q\beta}(x),\varphi_{q\beta}(y))<\alpha(\phi_{q\beta}(x))=\alpha(\phi_{t_i}(x)),\quad\quad\forall i\in\mathbb{Z}.
$$
Clearly $t_i\to\pm\infty$ as $i\to\pm\infty$.
Moreover, since $x\in K$ and $0\leq q\beta\leq t$ one has
$$
t_{i+1}-t_i=\beta\leq\alpha(\phi_{q\beta}(x))=\alpha(\phi_{t_i}(x)),\quad\quad\forall i\in\mathbb{Z}.
$$
Likewise,
$$
|u_{i+1}-u_i|\leq\alpha(\phi_{u_i}(y)), \quad\quad\forall i\in\mathbb{Z}.
$$
Then, $y\in \phi_{[-1,1]}(x)$ by Lemma \ref{p1} and so $x$ and $y$ belong to the same orbit of $\phi$.
Since distinct points in $E$ (like $x$ and $y$) belong to distinct orbits,
we get a contradiction.
Therefore, $E$ is $(t,\alpha,K)$-separating.
It follows that
$card(E)\leq S(t,\alpha,K))$ and, since $card(E)=v(t_*, {\frac{\beta}2})$,
we get \eqref{poro}.

Now, we observe that
$$
v(t)\leq \sum_{n=1}^{[\frac{t}{\beta(t,\alpha,K)}]}v(n\beta(t,\alpha,K), {\frac{\beta(t,\alpha,K)}2})
$$
so \eqref{poro} implies
$$
v(t)\leq t\cdot\frac{S(t,\alpha,K)}{\beta(t,\alpha,K)},\quad\quad\forall t>0.
$$
Then,
\begin{eqnarray*}
\limsup_{t\to\infty}\frac{1}t\log v(t)&\leq& \limsup_{t\to\infty}\left(\frac{1}t\log t+\frac{1}t\log \left(\frac{S(t,\alpha,K)}{\beta(t,\alpha,K)}\right)\right)\\
&=&
\limsup_{t\to\infty}\frac{1}t\log \left(\frac{S(t,\alpha,K)}{\beta(t,\alpha,K)}\right)\\
&\leq& e(\phi)
\end{eqnarray*}
completing the proof.
\end{proof}

\section*{Acknowledgments}

\noindent
The first author would like to thank the School of Mathematical Sciences of the
Liaoning University, Shenyang, China, for its kindly hospitality during the preparation of this work.

\section*{Funding}

\noindent
YY was partially supported by the National Natural Science Foundation of China (Grant Nos. 12101281) and Scientific Research Foundation of Education Department of Liaoning Province, China (Grant Nos. JYTQN2023190).

\section*{Declaration of competing interest}

\noindent
There is no competing interest.

\section*{Data availability}

\noindent
No data was used for the research described in the article.

\end{document}